\documentclass[11pt]{article}
\usepackage{authblk}
\usepackage{amssymb}
\usepackage{amsmath,mathtools}
\usepackage{amsfonts}
\usepackage{amsthm}
\usepackage{dsfont}
\usepackage{amssymb}
\usepackage{mathpazo}
\usepackage{fullpage}
\usepackage{enumerate}
\usepackage[square,sort,comma,numbers]{natbib} 

\usepackage{titlesec}
\titleformat{\section}[hang]{\Large\bfseries\filright}{\thesection.}{.5em}{}
\titleformat{\subsection}[hang]{\large\bfseries\filright}{}{0em}{}
\titleformat{\subsubsection}[block]{\bfseries}{}{0em}{}

\theoremstyle{definition}
\newtheorem{question*}{Question}

\newcommand{\C}{\mathbb{C}}

\newcommand{\tr}{\mathrm{Tr}}

\theoremstyle{plain}
\newtheorem{theorem}{Theorem}
\newtheorem{lemma}[theorem]{Lemma}
\newtheorem{corollary}[theorem]{Corollary}
\newtheorem{proposition}[theorem]{Proposition}

\newtheorem{remark}[theorem]{Remark}

\theoremstyle{definition}
\newtheorem{definition}[theorem]{Definition}

\newtheorem{example}[theorem]{Example}

\begin{document}

\title{Operator Algebra Generalization of a Theorem of Watrous and Mixed Unitary Quantum Channels}

\author{David~W.~Kribs$^1$, Jeremy Levick$^1$, Rajesh Pereira$^1$, Mizanur Rahaman$^2$}

\affil{$^1$Department of Mathematics \& Statistics, University of Guelph, Guelph, ON, Canada N1G 2W1. Email {mailto:dkribs@uoguelph.ca}{dkribs@uoguelph.ca}, {mailto:jerlevick@gmail.com}{jerlevick@gmail.com}, {mailto:pereirar@uoguelph.ca}{pereirar@uoguelph.ca} }
\affil{$^2$ Univ Lyon, ENS Lyon, UCBL, CNRS, Inria, LIP, F-69342, Lyon Cedex 07, France. Email {mailto:mizanur.rahaman@ens-lyon.fr}{mizanur.rahaman@ens-lyon.fr}}



\date{\today}
\maketitle

\begin{abstract}
We establish an operator algebra generalization of Watrous' theorem \cite{watrous2009} on mixing unital quantum channels (completely positive trace-preserving maps) with the completely depolarizing channel, wherein the more general objects of focus become (finite-dimensional) von Neumann algebras, the unique trace preserving conditional expectation onto the algebra, the group of unitary operators in the commutant of the algebra, and the fixed point algebra of the channel. As an application, we obtain a result on the asymptotic theory of quantum channels, showing that all unital channels are eventually mixed unitary. We also discuss the special case of the diagonal algebra in detail, and draw connections to the theory of correlation matrices and Schur product maps.
\end{abstract}

\section{Introduction}

Quantum channels, which are mathematically described by completely positive trace-preserving maps, are central objects of study in quantum information theory \cite{holevo2019quantum,nielsen, paulsen,watrous2018}. The class of unital (or doubly bistochastic) channels are a class of particular interest, and amongst such channels the subclass of mixed unitary channels arise in almost every area of quantum information theory (see \cite{audenaert2008random,girard2020mixed,gurvits2003classical,ioannou2006computational,kribs2021nullspaces,LeeWatrous2020mixed} as entrance points into the corresponding literature).
Hence a basic topic in the theory of quantum channels and their applications is the determination of when or how close a unital channel is to being mixed unitary. A fundamental result in this direction is a theorem of Watrous \cite{watrous2009}, which shows that any unital channel that is properly averaged with the `completely depolarizing channel', the map that sends all quantum states to the maximally mixed state, can be written as a mixed unitary channel.

In this paper, we obtain a generalization of Watrous' Theorem to the setting of  operator algebras. The more general objects of focus become (finite-dimensional) von Neumann algebras, the unique trace preserving conditional expectation onto the algebra, the group of unitary operators in the commutant of the algebra, and the fixed point algebra of the channel. The original theorem is recovered when applied to the special case of the (trivial) scalar algebra, wherein the completely depolarizing channel is viewed as the conditional expectation onto the algebra. Our proof is necessarily more intricate, requiring a number of supporting results that may be of independent interest.
As an application, we obtain a result on the asymptotic theory of quantum channels, and we show that all unital channels are eventually mixed unitary. We first show this for primitive unital channels using the Watrous theorem, and then we prove the general result following some prepatory work on irreducible unital channels and their peripheral eigenvalue algebras before applying the theorem.
Finally, the case of the diagonal algebra yields a connection with correlation matrices and Schur product maps, and we conclude by considering this case in more detail, interpreting the results in that setting and providing alternative viewpoints of the main theorem.

This paper is organized as follows. The next section includes requisite preliminary notions, and we motivate and formulate the main theorem statement. Section~3 includes the theorem proof, Section~4 derives the application discussed above, and Section~5 gives the detailed treatement of the diagonal algebra case.

\section{Background}

We begin by recalling basic preliminary notions, and then we formulate our main theorem.

\subsection{Preliminaries}

Given a positive integer $d \geq 1$, we let $M_d$ denote the set of $d \times d$ complex matrices. The matrix units $E_{ij}$, for $1 \leq i,j \leq d$, are the elements of $M_d$ with a 1 in the $i,j$ entry and 0's elsewhere. The (Hilbert-Schmidt) trace inner product on $M_d$ is given by $\langle A,B \rangle = \tr (B^* A)$. The tensor product algebra $M_d \otimes M_d$ is naturally identified with $M_{d^2}$ and has matrix units $E_{ij}\otimes E_{kl}$. We will make use of the linear map $\mathrm{vec}:M_d \rightarrow \C^d \otimes \C^d$ defined by $\mathrm{vec}(E_{ij}) = e_i \otimes e_j$, where $\{ e_1, \ldots ,e_{d} \}$ is the standard basis for $\C^d$.

We will be interested in completely positive maps $\Phi : M_d \rightarrow M_d$ \cite{paulsen}, which can always be represented in operator-sum form $\Phi(X) = \sum_i K_i X K_i^*$ for some set of `Kraus' operators $K_i \in M_d$ \cite{kraus}. Each map defines a dual map via the trace inner product, wherein the roles of the operators $K_i$ and $K_i^*$ are reversed in the operator-sum form.
The map $\Phi$ is unital if $\Phi(I)=I$, where $I$ is the identity matrix. If it is trace-preserving, which occurs exactly when $\sum_i K_i^* K_i = I$, then the map is called a {\it quantum channel} \cite{holevo2019quantum,nielsen,watrous2018}. The class of unital (quantum) channels are pervasive in quantum information, and we define an important subclass below. Note that a channel is unital if and only its dual map is a unital channel as well.

The `Choi matrix'  \cite{Choi1975} for $\Phi$ is the matrix $J(\Phi) \in M_d \otimes M_d$ given by,
\begin{equation}\label{choimatrix}
J(\Phi) = \sum_{i,j=1}^d E_{ij}\otimes \Phi(E_{ij}) .
\end{equation}
It is a positive semi-definite matrix if and only if $\Phi$ is a completely positive map. The map $J(\cdot)$ is linear, and we note that $J(\Phi) = \sum_k \mathrm{vec}(K_i) \mathrm{vec}(K_i)^*$ when the $K_i$ are Kraus operators for $\Phi$ \cite{watrous2018}.

By an {\it operator algebra}, we will mean a finite-dimensional von Neumann algebra (or C$^*$-algebra), which, up to unitarily equivalence \cite{davidson}, is a set of matrices contained inside some $M_d$ of the form:
\begin{equation}\label{vnalg}
\mathcal A = \oplus_{k} (I_{m_k} \otimes M_{n_k}),
\end{equation}
for some unique choice of positive integers $m_k, n_k$. The algebras we consider will typically be the fixed point sets of unital channels, and so necessarily will be unital ($I\in \mathcal A$), which means that $\sum_k m_k n_k = d$.  The commutant $\mathcal A^\prime$ of $\mathcal A$, which is the set of all matrices in $M_d$ that commute with every element of $\mathcal A$, has a corresponding form up to unitary equivalence given by $\mathcal A^\prime = \oplus_k (M_{m_k} \otimes I_{n_k})$.

Given an algebra $\mathcal A \subseteq M_d$, we can consider {\it conditional expectations} onto the algebra, which are maps $\mathcal E : M_d \rightarrow \mathcal A$ such that: (1) $\mathcal E(A)=A$ for all $A\in \mathcal A$; (2) $\mathcal E( A_1 X A_2) = A_1 \mathcal E(X) A_2$ for all $A_1,A_2\in  \mathcal A$ and $X\in M_d$; and, (3) if $X\in M_d$ is positive semi-definite, then so is $\mathcal E(X)$.
Every conditional expectation of $M_d$ onto $\mathcal A$ is completely positive (and is a unital map when the algebra is unital), and amongst all possible conditional expectations onto $\mathcal A$, there is a unique map that is also trace-preserving \cite{pereira2006representing} (in fact, it is exactly the orthogonal projection of $M_d$ onto $\mathcal A$ in the trace inner product). So given a unital algebra $\mathcal A \subseteq M_d$, we shall denote the trace preserving conditional expectation onto $\mathcal A$ by $\mathcal E_{\mathcal A} : M_d \rightarrow \mathcal A$.

The fixed point set $\mathrm{Fix}(\Phi) = \{ X \in M_d \, | \, \Phi(X) = X \}$ will also play a key role in our analysis. For a unital map $\Phi$, it is easily seen that $\mathrm{Fix}(\Phi)$ contains the commutant of the Kraus operators, and further, for a unital channel these two sets coincide; $\mathrm{Fix}(\Phi) = \{K_i \}^\prime$ \cite{kribscommut}. In particular, this means the fixed point set, which in general is just an operator subspace, in the case of unital channels is an operator algebra.

We shall focus on the following class of unital channels, which are important in several areas of quantum information \cite{audenaert2008random,girard2020mixed,gurvits2003classical,ioannou2006computational,kribs2021nullspaces,LeeWatrous2020mixed}.

\begin{definition}
A completely positive linear map $\Phi: M_d \rightarrow M_d$ is called a {\it mixed unitary channel} if there exists a set of $d \times d$ unitary matrices $\{U_i\}_{i=1}^r$ and a set of nonnegative numbers $\{\lambda_i\}_{i=1}^r$ with $\sum_{i=1}^r \lambda_i=1$ such that $\Phi(X)=\sum_{i=1}^r \lambda_i U_iXU_i^*$.
\end{definition}

It is known that every single-qubit unital channel is mixed unitary, but this is not the case for higher dimensions.

\subsection{Formulation of the Conjecture}

We shall establish a generalization of the following theorem of Watrous \cite{watrous2009} to the setting of operator algebras.

\begin{theorem}\label{watrous}
Let $\Phi: M_d \rightarrow M_d$ be a unital quantum channel. Then for $0 \leq p \leq 1/ (d^2-1)$, the convex combination of maps given by
\begin{equation}
p \, \Phi(X) + (1-p) \frac{\tr(X)}{d} I_d
\end{equation}
is a mixed unitary channel.
\end{theorem}

The {\it completely depolarizing channel} $\delta_d : M_d \rightarrow M_d$ is defined as:
$\delta_d(X) = d^{-1}\tr (X) I_d$.
As it is also a mixed unitary channel (implemented by any set of (uniformly scaled) unitary operators that form an orthogonal basis in the trace inner product on $M_d$), and the set of mixed unitary channels is convex, the theorem is proved by explicitly proving the case $p = 1/(d^2-1)$. That is,  Theorem~\ref{watrous} is equivalent to the statement that the convex combination $p \Phi + (1-p)\delta_d$ is mixed unitary for $p = \frac{1}{d^2-1}$.

Some initial investigation shows that a naive generalization of Theorem~\ref{watrous} does not hold. Consider the following example, which illustrates this point.

\begin{example}\label{watrouseg}
We have the identity map $\mathrm{id}: M_3\rightarrow M_3$, $\mathrm{id}(X)=X$, and the Werner-Holevo channel \cite{werner2002counterexample} on $M_3$ given by,
\[
W_3^{-}(X)=\frac{1}{2} [\tr(X)-X^t],
\]
where $X^t$ is the transpose of $X$. Now consider the channel $\Phi_p : M_3 \rightarrow M_3$ for $0<p\leq 1$ given by,
\[
\Psi_p(X)=p X+ (1-p) W_3^{-}(X).
\]
We claim that this channel is not mixed-unitary for any $p$. Indeed, first note that any operator-sum representation of $\Psi_p$ will have Kraus operators of the form $K=\alpha I+ A$, where $A$ is an anti-symmetric matrix, $I$ is the identity matrix and $\alpha$ is a constant. (To see this, observe that $\Psi_p$ has a representation of this form, and then note this implies any representation has this form.) As $A$ is a $3\times 3$ anti-symmetric matrix, the eigenvalues are $\lambda, -\lambda, 0$. So the eigenvalues of $K$ are $\alpha+\lambda, \alpha-\lambda, \alpha $. Now if $K$ is a multiple of a unitary, then these eigenvalues must lie on a circle. However, these three numbers are co-linear. Hence $\lambda=0$ and so $A=0$. This is true for all Kraus operators and thus it follows that $p=1$.

One might expect a naive generalization of the original Watrous Theorem to find that $t \mathrm{id} + (1-t)\Phi$ is mixed unitary for some $t$, simply replacing $\delta_3$ with the identity map $\mathrm{id}$. But in fact, $W_3^{-}$ is a channel for which $t \mathrm{id} + (1-t) W_3^{-}$ is not mixed unitary for any $t < 1$. So, simply replacing the depolarizing channel by another unital channel immediately yields that there are channels for which no non-trivial convex combination is mixed unitary.
\end{example}

After some more thought, we were led to view Theorem~\ref{watrous} as a special case of a more general phenomena in the context of operator algebras. In particular, in seeking to generalize the theorem, we make the following observations:
\begin{itemize}
\item $\delta_d$ is the (unique) trace preserving conditional expectation onto the trivial scalar algebra $\mathcal A = \C I_d$.
\item Every unital channel $\Phi$ contains the trivial algebra in its fixed point algebra; $\mathrm{Fix}(\mathcal A) \supseteq \C I_d$.
\item The unitary group $\mathcal{U}(d)$ inside $M_d$ is the group of unitaries contained in the commutant of the trivial algebra; $(\C I_d)^\prime = M_d$.
\end{itemize}

Following further investigation, we replace the trivial algebra $\C I_d$ with an arbitrary unital operator algebra $\mathcal{A}$, and then we formulated a conjecture on the generalization, which we state and prove as the following result.

\begin{theorem}\label{main}
Let $\mathcal{A}$ be any unital operator algebra inside $M_d$. Let $\mathcal{E}_{\mathcal{A}}$ be the trace preserving conditional expectation onto $\mathcal{A}$, and let $\mathcal{U}_{\mathcal{A}'}$ be the group of unitaries contained in the commutant of $\mathcal{A}$. Then for any unital channel $\Phi$ whose fixed point algebra contains $\mathcal{A}$, there exists a $p \in (0,1)$ depending only on the algebra $\mathcal{A}$ such that the convex combination
\begin{equation}
p\Phi + (1-p)\mathcal{E}_{\mathcal{A}}\end{equation} is in the convex hull of channels of the form $\Phi_U (X) = UXU^*$ where $U\in \mathcal{U}_{\mathcal{A}'}$.
\end{theorem}

Returning to the example above, the point is, in order to generalize properly, one must restrict the set of channels, $\Phi$, to only those that fix the algebra onto which the conditional expectation projects.
In the example, $\mathrm{id}$ is the conditional expectation onto the full matrix algebra, $M_3$, and in fact there are no non-trivial channels that fix this algebra. It is also the case that the fixed point algebra of $W_3^{-}(X)$ is just the trivial algebra, consisting of scalar multiples of the identity matrix. Thus, there is no unital channel other than $\delta_3$ for which we should expect a Watrous-type theorem to hold for $W_3^{-}$.


\section{Proof of Main Result}

In this section we shall prove Theorem~\ref{main}. The proof requires a number of supporting results that may be of independent interest. We begin by establishing notation.

Let $d = \sum_{k=1}^r m_k n_k$ for some positive integers $m_k, n_k$, and let $D = \sum_{k=1}^r n_k^2$. For the purposes of the proof, here we will assume the algebra $\mathcal A$ is given by,
$$
\mathcal{A} = \oplus_{k=1}^r M_{m_k}\otimes I_{n_k},
$$
so that
$$
\mathcal{A}' = \oplus_{k=1}^r I_{m_k}\otimes M_{n_k},
$$
and note this means the vector space dimension of the commutant is $D = \dim(\mathcal{A}')$.

Let $\{K_i\}_{i=1}^n$ be a fixed set of Kraus operators for $\Phi$. Then by assumption we have $\mathcal A \subseteq \mathrm{Fix}(\Phi) = \{ K_i \}^\prime$, so that the $K_i$ belong to $\mathcal A^\prime$ and hence each $K_i = \oplus_{k=1}^r I_{m_k}\otimes K_{ik}$ for some $K_{ik} \in M_{n_k}$.  Define $\Phi_k$, for each $k$, to be the map on $M_{n_k}$ with Kraus operators $\{K_{ik}\}_{i=1}^n$, and define $\widehat{\Phi_k}$ to be the map on $M_d$ whose Kraus operators $\widehat{K_{ik}}$ have $I_{m_k}\otimes K_{ik}$ on the $k$th block and zeroes on the other blocks. Note that $\Phi_k$ is a unital channel as $\Phi$ is.

We consider unitaries $U\in\mathcal A^\prime$, which are of the form $U = \oplus_{k=1}^r I_{m_k}\otimes U_k$ with $U_k \in \mathcal{U}(n_k)$.
For $A,B \in \mathcal{A}'$ define the inner product:
$$
\langle A,B\rangle_{\mathcal{A}'} = \sum_{k=1}^r n_k \tr(A_k^*B_k),
$$
where $A = \oplus_{k=1}^r I_{m_k}\otimes A_k$ and similarly for $B$. This is the inner product that arises from the left-regular representation of $\mathcal A^\prime$ \cite{jacobson1985basic,pereira2006representing}. 

Further, let $\delta_{n_k}$ be the depolarizing map on $M_{n_k}$, and recall we have $\mathcal E_{\mathcal{A}}$ as the trace-preserving conditional expectation onto $\mathcal{A}$.

Given any completely positive map $\Phi: M_d \rightarrow M_d$ with Kraus operators $K_i \in \mathcal{A}'$, define the following linear map on $M_d$:
\begin{equation}
L(\Phi)(X) = \sum_i\int_{U \in \mathcal{U}(\mathcal{A}')} UXU^* \bigl|\langle U, K_i\rangle_{\mathcal{A}'}\bigr|^2 d\mu(U)  \quad \quad \forall X \in M_d ,
\end{equation}
where $\mu(\cdot)$ is the Haar measure on the unitary group $\mathcal U (\mathcal A^\prime) = \mathcal A^\prime \cap \mathcal U(d)$. Notice that $L(\Phi)$ is a positive combination of unitary adjunctions, so possibly after some normalizing, $L(\Phi)$ is a mixed-unitary map. We note that $L(\cdot)$ depends on the algebra $\mathcal A$, though we will suppress reference to it in the notation, and also observe that $L(\Phi + \Psi) = L(\Phi) + L(\Psi)$ for any completely positive maps with Kraus operators in $\mathcal A^\prime$. 

We collect the following known results (with short proofs for completeness) before analyzing the map $L(\Phi)$ in more detail.

\begin{lemma}\label{unitaryprojint}
For any positive integer $d$, we have
\begin{equation}\label{unitaryprojinteqn}
\int_{U \in \mathcal{U}(d)} \mathrm{vec}(U)\mathrm{vec}(U)^* d\mu(U) = \frac{1}{d}I_d \otimes I_d.
\end{equation}
\end{lemma}

\begin{proof}
For $\Phi(X) = UXU^*$, the Choi matrix is $J(\Phi) = \mathrm{vec}(U)\mathrm{vec}(U)^*$. Hence, the above integral is simply the Choi matrix of the channel
\[
\delta_d(X) = \int_{U\in \mathcal{U}(d)} UXU^* d\mu(U) = \frac{1}{d}\tr (X)I_d,
\]
and it is clear that $\frac{1}{d}I_d \otimes I_d$ is the correct Choi matrix.
\end{proof}

Since $U\otimes U^*$ has the same entries as $\mathrm{vec}(U)\mathrm{vec}(U)^*$, up to the permutation that maps $E_{ij}\otimes E_{lk} \mapsto  E_{ik}\otimes E_{jl}$, and since this same permutation maps $\frac{1}{d}I_d\otimes I_d \mapsto  \frac{1}{d}\sum_{i,j} E_{ij}\otimes E_{ji}$, we also have the following.

\begin{corollary} For any positive integer $d$, we have
\begin{equation}\label{unitarytensorint}
\int_{U\in \mathcal{U}(d)} U\otimes U^* d \mu(U) = \frac{1}{d} \sum_{i,j=1}^d E_{ij}\otimes E_{ji}.
\end{equation}
\end{corollary}

This in turn, gives us another Corollary that will be useful.

\begin{corollary}\label{integraltrace}
For any $X \in M_d(\C)$, we have
\begin{align}
\label{intXwithU} \int_{U\in \mathcal{U}(d)} U \, \tr (U^*X)d\mu(U) & = \frac{1}{d}X \\
\label{intXwithUsq} \int_{U\in\mathcal{U}(d)} |\tr(UX^*)|^2 d\mu(U) &= \frac{1}{d} \tr (X^*X).
\end{align}
\end{corollary}

\begin{proof}
This follows from the fact that the two integrals can be expressed as $(\mathrm{id}\otimes\tr)\bigl(P(I_d\otimes X)\bigr),$ and $\tr(P(X^*\otimes X))$ respectively, where $P = \int_{U \in \mathcal{U}(d)} U\otimes U^* d\mu(U)$. By Equation \ref{unitarytensorint} this is just $\frac{1}{d}\sum_{i,j} E_{ij}\otimes E_{ji}$, and so we get, respectively,
\[
\frac{1}{d} \sum_{i,j} \tr (XE_{ji}) E_{ij} = \frac{1}{d}X, \quad \mathrm{and} \quad \frac{1}{d}\sum_{i,j} \tr (X^*E_{ij})\tr (XE_{ji}) = \frac{1}{d}\sum_{i,j}|x_{ij}|^2,
\]
where $X = (x_{ij})$, which completes the proof.
\end{proof}

Using these facts, in the next pair of results we can derive useful properties of the $L(\cdot)$ map.
Before beginning the proofs in earnest, as preparation we briefly discuss the Haar integral over the group $\mathcal U(\mathcal A^\prime) = \mathcal A^\prime \cap \mathcal U(d)$ and explain some facts that we will use to simplify expressions in the analysis below. First of all, the group of unitaries in $\mathcal{A}'$ is, as a group, simply the product of the groups $I_{m_k}\otimes \mathcal U(n_k)$. The Haar measure on these component groups is just the Haar measure on each $\mathcal U(n_k)$, and so the Haar measure on the finite product of these groups is just the product of these Haar measures; thus we have 
\[
\int_{U\in \mathcal U(\mathcal{A}')} X d\mu(U) = \int_{U_1 \in \mathcal U(n_1)}\cdots \int_{U_r \in \mathcal U(n_r)} X d\mu(U_r) \cdots d\mu(U_1), 
\] 
and indeed the right-hand-side can be arranged into any permutation of the groups $\mathcal U(n_j)$ \cite{hewittross}. 
In our evaluation of integrals below, our integrand will be an expression containing only a small number of the $U_j$; we will as a matter of course rewrite all integrals so that integrals over $\mathcal U(n_j)$ where there is no appearance of $U_j$ or $U_j^*$ inside the integrals become the innermost integrals. This is because such integrals will reduce to trivial integrals; and since we then integrate over normalized Haar measure, these inner integrals integrate to $1$, and thus no longer appear explicitly. We will do all of this implicitly, so as not to clutter notation. For example, 
\begin{align*}
\int_{U\in \mathcal U(\mathcal{A}')} U_1 d\mu(U) &= \int_{U_1\in \mathcal U(n_1)}U_1 \biggl(\cdots \int_{U_{r-1}\in \mathcal U(n_{r-1})}\underbrace{\biggl[\int_{U_r \in \mathcal U(n_r)} d\mu(U_r)\biggr]}_{=1}d\mu(U_{r-1})\cdots \biggr) d\mu(U_1)  \\
& = \int_{U_1 \in \mathcal U(n_1)} \biggl(\cdots \underbrace{ \int_{U_{r-1}\in \mathcal U(n_{r-1})} 1 d\mu(U_{r-1})}_{=1}\cdots \biggr) d\mu(U_1) \\
&= \int_{U_1 \in \mathcal U(n_1)} U_1 d \mu(U_1).
\end{align*}
Thus, from here on out, we will immediately jump to the simplified form, and all integrals will only be taken over variables that actually appear in a non-trivial way in any given expression. Finally, even for the remaining variables, we will only leave one integral sign, to avoid clutter and confusion; for instance, if the variables $U_j, U_k$ appear inside an integration, the expression $\int f(U_j,U_k)d\mu(U_j)d\mu(U_k)$ should be understood as $\int_{U_k \in \mathcal U(n_k)}\int_{U_j\in \mathcal U(n_k)} f(U_j,U_k) d\mu(U_j)d\mu(U_k)$. 

\begin{lemma}\label{LofPhi}
Suppose $\Phi : M_d \rightarrow M_d$ is a unital channel that fixes the algebra $\mathcal A$. Then for all $X\in M_d$, 
$$
L(\Phi)(X) = \Phi(X) + (D-1)\mathcal E_{\mathcal{A}}(X) + \sum_{k: n_k>1} \frac{1}{n_k^2-1}\bigl(\widehat{\Phi_k}(X) - \widehat{\delta_{n_k}}(X)\bigr).
$$
\end{lemma}

\begin{proof}
We first expand, using the following form for a generic element of $M_d$:
\[
X = \sum_{k_1,k_2,s,t} E_{k_1 k_2}\otimes E_{st}\otimes X_{k_1k_2,st},
\]
where for each pair $1 \leq k_1, k_2 \leq r$, the matrices $E_{st}$, with $1\leq s \leq m_{k_1}$ and $1\leq t \leq m_{k_2}$ are matrix units for the $m_{k_1} \times m_{k_2}$ matrices, and $X_{k_1k_2,st}$ is a $n_{k_1}\times n_{k_2}$ matrix. So in this form, $U\in \mathcal A^\prime$ is written $U = \sum_{k=1}^r E_{kk} \otimes I_{m_k} \otimes U_k$ with $U_k\in M_{n_k}$.

From the definition of $L$, and using the expansion of $X$, also keeping in mind the integral note above, we have,
\begin{eqnarray*}
L(\Phi)(X) &=& \sum_i \sum_{k_1,k_2,s,t,j,l} E_{k_1 k_2}\otimes E_{st} \otimes \\ 
& &  \int U_{k_1} X_{k_1k_2,st}U_{k_2}^* n_j n_l\tr (U_j^*K_{ij})\tr (U_lK_{il}^*)d\mu(U_{k_1})d\mu(U_{k_2})d\mu(U_j)d\mu(U_l) ,
\end{eqnarray*}
and so we will analyze this depending on whether or not $k_1 = k_2$.

If $k_1 \neq k_2$, then, using the fact that the Haar integral satisfies $\int_U U \, d\mu(U) =0$, non-zero contributions in this expression can only come from the cases $j=k_1$ and $l=k_2$ where we get using Corollary~\ref{integraltrace},
\begin{align*}
\sum_{k_1,k_2,s,t} E_{k_1k_2}\otimes E_{st} \otimes \biggl(\sum_i \bigl(n_{k_1} \int U_{k_1} \tr (U_{k_1}^*K_{i k_1}) d\mu(U_{k_1})\bigr)X_{k_1k_2,st}\bigl( n_{k_2} \int \bigl( U_{k_2} \tr (U_{k_2}^*K_{i k_2})\bigr)^* d\mu(U_{k_2})  \bigr) \biggr) \\
= \sum_{k_1,k_2,s,t} E_{k_1 k_2}\otimes E_{st} \otimes\biggl(\sum_i K_{ik_1}X_{k_1k_2,st}K_{ik_2}^*\biggr),
\end{align*}
which gives us the off-diagonal blocks of
$\Phi(X) = \sum_i K_i X K_i^*.$
Thus, on the blocks corresponding to $k_1 \neq k_2$, we get a term of the corresponding block form of $\Phi(X)$.

When $k_1= k_2$, we must have $j=l$ for non-zero contributions, and we split this up into terms for which $k_1\neq j$ and $k_1=j$ to get,
\begin{align*}
\sum_{k_1,s,t} E_{k_1k_1}\otimes E_{st}  \otimes \biggl(\bigl(\int U_{k_1} X_{k_1k_1,st}U_{k_1}^*d\mu(U_{k_1})\bigr)\bigl(\sum_{j\neq k_1}n_j^2\sum_i \int \bigl|\tr (U_j^*K_{ij})\bigr|^2d\mu(U_j)\bigr) \biggr) \\
+ \sum_{k_1,s,t} E_{k_1k_1}\otimes E_{st} \otimes n_{k_1}^2\sum_i \int U_{k_1} X_{k_1k_1,st}U_{k_1}^*\bigl|\tr (U_{k_1}^*K_{ik_1})\bigr|^2 d\mu(U_{k_1}) . 
\end{align*}
Now, using the definition of the completely depolarizing channel on $M_{n_k}$, Equation~\ref{intXwithUsq}, and the fact that $\{ K_{ij} \}_i$ defines a channel on $M_{n_j}$ applied to the first term, and then the estimate from the original Watrous Theorem applied to the second term (and assuming for now that $n_{k_1}\neq 1$), we obtain the following:
$$
\sum_{k_1,s,t} E_{k_1k_1}\otimes E_{st} \otimes \biggl(\delta_{n_{k_1}}(X_{k_1k_1,st})\sum_{j\neq k_1}n_j^2 + \frac{n_{k_1}^2}{n_{k_1}^2-1}\Phi_{k_1}(X_{k_1k_1,st}) + \frac{n_{k_1}^2(n_{k_1}^2-2)}{n_{k_1}^2-1}\delta_{n_{k_1}}(X_{k_1k_1,st})\biggr).
$$
The last tensor factor can be rewritten as, 
\begin{align*}\delta_{n_{k_1}}(X_{k_1k_1,st})(D - n_{k_1}^2) + \Phi_{k_1} (X_{k_1k_1,st}) +\frac{1}{n_{k_1}^2-1}\Phi_{k_1}(X_{k_1k_1,st}) + \frac{n_{k_1}^2(n_{k_1}^2-2)}{n_{k_1}^2-1}\delta_{n_{k_1}}(X_{k_1k_1,st}) \\
= \Phi_{k_1}(X_{k_1k_1,st}) + D \delta_{n_{k_1}}(X_{k_1k_1,st}) + \frac{1}{n_{k_1}^2-1}\Phi_{k_1}(X_{k_1k_1,st}) - \frac{n_{k_1}^2}{n_{k_1}^2-1}\delta_{n_{k_1}}(X_{k_1k_1,st});\end{align*}
which follows from splitting $\frac{n_{k_1}^2}{n_{k_1}^2-1}$ into $1 + \frac{1}{n_{k_1}^2-1}$ and $\frac{n_{k_1}^2(n_{k_1}^2-2)}{n_{k_1}^2-1} - n_{k_1}^2 = -\frac{n_{k_1}^2}{n_{k_1}^2-1}$. This latter quantity, in turn, we can write as $-\frac{1}{n_{k_1}^2-1} - 1$, and so we have the last tensor factor rewritten as, 
\[
\Phi_{k_1}(X_{k_1k_1,st}) + (D-1) \delta_{n_{k_1}}(X_{k_1k_1,st}) + \frac{1}{n_{k_1}^2-1}\biggl( \Phi_{k_1}(X_{k_1k_1,st}) - \delta_{n_{k_1}}(X_{k_1k_1,st})\biggr) . 
\]

We must also consider the case when $n_{k_1} = 1$, in which case we cannot use the Watrous Theorem as written, to avoid dividing by $n_{k_1}^2-1=0$. In this case, note that $\Phi_{k_1}$ is a channel on $M_1 \equiv \C$; but there is only one such map, which is the identity map. In this case, $\Phi_{k_1} = \delta_{n_{k_1}}$ is just the identity map on $\C$, and so in that case, we would write the relevant integral in the last tensor factor as, 
\[
n_{k_1}^2 \sum_i\int U_{k_1} X_{k_1k_1,st}U_{k_1}^* |\tr (U_{k_1}^*K_{i k_1})|^2 = X_{k_1k_1,st} = \Phi_{k_1}(X_{k_1k_1,st}) = \delta_{n_{k_1}}(X_{k_1k_1,st}).
\]
So in this case, the diagonal term would be
\[
(D-1)\delta_{n_{k_1}}(X_{k_1k_1,st}) + \Phi_{k_1}(X_{k_1k_1,st}).
\]

Hence, summing over $k_1, k_2$ in the decomposition of $L(\Phi)(X)$, we first get a copy of $\Phi(X)$. Further, the $(D-1)\delta_{n_{k_1}}(X_{k_1k_1,st})$ down the diagonal combine to give us $(D-1) \mathcal E_{\mathcal{A}}(X)$. Thus, bringing everything together, we have
$$
L(\Phi)(X) = \Phi(X) + (D-1) \mathcal E_{\mathcal{A}}(X) +\sum_{k:n_{k}>1} \frac{1}{n_{k}^2-1}\bigl(\widehat{\Phi_{k}}(X) - \widehat{\delta_{n_{k}}}(X)\bigr),
$$
and this completes the proof.
\end{proof}

\begin{lemma}\label{LofPhihat}
Given $\mathcal A$ and $\Phi$ as above, for each unital channel $\Phi_k$, we have for all $X\in M_d$, 
$$
L(\widehat{\Phi_k})(X) = n_k^2 \mathcal E_{\mathcal{A}}(X) + \frac{n_k^2}{n_k^2-1}\bigl(\widehat{\Phi_k}(X) - \widehat{\delta_{n_k}}(X)\bigr), 
$$ 
if $n_k>1$, and otherwise,
$$
L(\widehat{\Phi_k})(X) = \mathcal E_{\mathcal{A}}(X).
$$
\end{lemma}

\begin{proof}
Notice that $\langle U, \widehat{K_{i k}}\rangle_{\mathcal{A}'} = n_k \tr (U_k^*K_{ik})$ for all $i$ since all other blocks of $\widehat{K_{ik}}$ are $0$. Thus when we do the same calculation as in the previous Lemma proof with $\widehat{\Phi_{k}}$ instead of $\Phi$, we get in the last tensor factor, 
$$
\sum_i \int U_{k_1} X_{k_1 k_2 ,st} U_{k_2}^* n_k^2|\tr(U_k^*K_{ik})|^2d\mu(U_{k_1})d\mu(U_k)d\mu(U_{k_2}) . 
$$ 
For $k_1 \neq k_2$, to get a non-zero contribution we must have in the integral, $k_1=k$ and $k_2=k$, a contradiction. So we only get (potentially) non-zero terms when $k_1=k_2$. 

If $k_1=k_2 \neq k$, we get
$$
\biggl( \int U_{k_1} X_{k_1 k_1,st}U_{k_1}^*d\mu(U_{k_1})\biggr) \biggl( n_k^2 \sum_i \int |\tr (U_k^*K_{ik})|^2d\mu(U_k)\biggr) ,
$$
and this simplifies, from the definition of the completely depolarizing channel applied to the first term and Corollary~\ref{integraltrace} and that $\{ K_{ik}\}_i$ define a channel on $M_{n_k}$ applied to the second term, to:
$$
n_k^2 \delta_{n_{k_1}}(X_{k_1 k_1,st}).
$$

Otherwise, if $k_1=k = k_2$, and $n_k>1$, we get using the original Watrous Theorem, 
$$
\frac{n_k^2}{n_k^2-1}\biggl(\Phi_k(X_{kk,st}) + (n_k^2-2)\delta_{n_k}(X_{kk,st})\biggr) . 
$$
From this, we take a term of the form $n_k^2 \delta_{n_k}(X_{kk,st})$ to combine with the other diagonal terms, giving us $n_k^2 \mathcal E_{\mathcal{A}}(X)$, and leaving us with
$$
\frac{n_k^2}{n_k^2-1}\bigl(\widehat{\Phi_k}(X) - \widehat{\delta_{n_k}}(X)\bigr)
$$
from the remaining terms.

Finally, if $n_k=1$, we just get
$$
n_k^2  \delta_{n_k}(X_{kk,st})
$$
since both $\delta_{n_k}$ and $\Phi_k$ are just the identity map on scalars.

So, we can sum over $k_1,k_2$ and the $n_k^2 \delta_{n_{k_1}}(X_{k_1 k_1,st})$ terms will combine to give us $n_k^2 \mathcal E_{\mathcal{A}}(X)$, plus, if $n_k>1$, a term of the form $\frac{n_k^2}{n_k^2-1}\bigl(\widehat{\Phi_k} - \widehat{\delta_{n_k}}\bigr)(X)$.
\end{proof}

Applying the previous Lemma to the case $\Phi_k = \delta_{n_k}$, we immediately have the following: 
\begin{equation}\label{Lofdeltahat} 
L(\widehat{\delta_{n_k}})(X) = n_k^2\mathcal E_{\mathcal{A}}(X) .
\end{equation}

Now we prove the theorem, with explicit constants given.

\begin{theorem}
Let $\Phi : M_d \rightarrow M_d$ be a unital channel that fixes the unital algebra $\mathcal{A}=\oplus_{i=1}^r M_{m_k}\otimes I_{n_k}$. Let $D = \dim (\mathcal{A}')$ and let $\widehat{r}$ be the number of direct summands of $\mathcal{A}'$ for which the matrix component satisfies $n_i > 1$.

If $r=1$, and $\mathcal{A}' = I_{m_1} \otimes M_{n_1}$, with $n_1 > 1$, we have that
\[
\frac{1}{n_1^2-1}\bigl(\Phi(X) + (n_1^2-2)\mathcal E_{\mathcal{A}}(X)\bigr)
\]
is mixed unitary.
If $r>1$, we have that
\[\frac{1}{D - \widehat{r} + \sum_{k: n_k > 1} n_k^2}\bigl(\Phi(X) + (D - \widehat{r} - 1 + \sum_{k : n_k>1} n_k^2)\mathcal E_{\mathcal{A}}(X)\bigr)
\]
is mixed unitary.
\end{theorem}

\begin{proof}
We begin with the statement of Lemma \ref{LofPhi}, combined with that of Lemma \ref{LofPhihat}:
$$
L(\Phi)(X) = \Phi(X) + (D-1)\mathcal E_{\mathcal{A}}(X) + \sum_{k: n_k >1} \frac{1}{n_k^2-1}\bigl(\widehat{\Phi_k} - \widehat{\delta_{n_k}}\bigr)(X) , 
$$
and
$$
L(\widehat{\Phi_k})(X) = n_k^2 \mathcal E_{\mathcal{A}}(X) + \frac{n_k^2}{n_k^2-1}\bigl(\widehat{\Phi} - \widehat{\delta_{n_k}}\bigr)(X)
$$
for all $n_k > 1$. Thus we can write
$$
L(\Phi)(X) = \Phi(X) + (D-1)\mathcal E_{\mathcal{A}}(X) + \sum_{k : n_k >1}\biggl(\frac{1}{n_k^2}L(\widehat{\Phi_k})(X) - \mathcal E_{\mathcal{A}}(X)\biggr) ,
$$
which we can rearrange to obtain
\begin{equation}\label{Lsrearranged}
L(\Phi)(X) - \sum_{k : n_k > 1} \frac{1}{n_k^2}L(\widehat{\Phi_k})(X) = \Phi(X) + (D - \widehat{r} - 1)\mathcal E_{\mathcal{A}}(X).
\end{equation}

Before continuing, we first consider the case where $r=1$; that is, where $\mathcal{A} = M_m \otimes I_n$ and $\mathcal{A}' = I_m\otimes M_n$ (and $n> 1$, as the $r=1=n$ case is trivial). Then $\widehat{r} = 1$, and we have that $\Phi_1 = \widehat{\Phi_1} = \Phi$, since there are no blocks other than the first block to zero out, in order to create $\widehat{\Phi_1}$. In this case, $n_1 = n$, and so we see that $L(\Phi) - \frac{1}{n^2}L(\widehat{\Phi_1})$ is just $\frac{n^2-1}{n^2}L(\Phi)$, and, using Equation \ref{Lsrearranged} (and $D=n^2$), we have
\[
\frac{n^2-1}{n^2}L(\Phi)(X) = \Phi(X) + (n^2-2)\mathcal E_{\mathcal{A}}(X).
\]

In any other case, for any $k$ for which $n_k > 1$, we have that $\Phi_k : M_{n_k}\rightarrow M_{n_k}$ is a unital channel.  
Thus $\tr (J(\Phi_k)) = n_k$, and since $J(\Phi_k) \geq 0$, we have that
$
n_k I_{n_k}\otimes I_{n_k} \geq J(\Phi_k).
$
Hence
\[
J(\delta_{n_k}) = \frac{1}{n_k} I_{n_k}\otimes I_{n_k} \geq \frac{1}{n_k^2}J(\Phi_k) ,
\]
and so the map $\delta_{n_k} - \frac{1}{n_k^2}\Phi_k$ is a completely positive map.
From this, we obtain that the map $\widehat{\delta_{n_k} - \frac{1}{n_k^2}\Phi_k} = \widehat{\delta_{n_k}} - \frac{1}{n_k^2}\widehat{\Phi_k}$ is also completely positive as it is just the direct sum of $\delta_{n_k} - \frac{1}{n_k^2}\Phi_k$ with the zero map a number of times.

Finally, using the fact that $L(\Phi + \Psi) - L(\Psi) = L(\Phi)$ for completely positive maps with Kraus operators in $\mathcal A^\prime$, we have that 
\[
L(\widehat{\delta_{n_k} - \frac{1}{n_k^2}\Phi_k})(X) = L(\widehat{\delta_{n_k}})(X) - \frac{1}{n_k^2} L(\widehat{\Phi_k})(X)
\]
is a positive combination of unitary adjunctions.  

Thus we may add $L(\widehat{\delta_{n_k}})(X)$ to both sides of Equation \ref{Lsrearranged}. By Equation \ref{Lofdeltahat}, this is equivalent to adding $n_k^2 \mathcal E_{\mathcal{A}}(X)$, and so we obtain
\[
L(\Phi)(X) + \sum_{k: n_k > 1} L(\widehat{\delta_{n_k} - \frac{1}{n_k^2}\Phi_k})(X) = \Phi(X) + \bigl(D - \widehat{r}-1 + \sum_{k : n_k >1} n_k^2 \bigr)\mathcal E_{\mathcal{A}}(X) .
\]

In all cases, we now have on the left-hand-side a positive combination of terms of the form $L(\Psi)$ where $\Psi$ is completely positive, and so this is a positive combination of unitary adjunctions. Hence the right-hand-side is now simply a positive combination of $\Phi$ and $\mathcal E_{\mathcal{A}}$. Thus, after suitably normalizing, the left-hand-side will be an expression of the right-hand-side as a mixed unitary. In particular, we obtain either
\[
\frac{1}{n^2}L(\Phi) = \frac{1}{n^2-1}\bigl(\Phi(X) + (n^2-2)\mathcal E_{\mathcal{A}}(X)\bigr)
\]
is mixed unitary in the case that $r=1$, or
\begin{align*}
\frac{1}{D-\widehat{r} + \sum_{k : n_k >1}n_k^2}\biggl(L(\Phi)(X) + \sum_{k : n_k > 1}L(\widehat{\delta_{n_k} - \frac{1}{n_k^2}\Phi_k})(X)\biggr) \\
= \frac{1}{D - \widehat{r} + \sum_{k: n_k >1}n_k^2} \biggl( \Phi(X) + (D - \widehat{r} - 1 + \sum_{k : n_k>1} n_k^2 )\mathcal E_{\mathcal{A}}(X)\biggr)
\end{align*}
is mixed unitary when $r > 1$. This completes the proof.
\end{proof}


\section{Application: All Unital Channels are Eventually Mixed Unitary}

In this section we prove that every unital quantum channel has the property that some power of it becomes mixed unitary. This involves proving several supporting results that may be of independent interest, and, at the final stage, applying our Theorem~\ref{main}. We note for the reader that following the logical flow of this section does not require the results of the previous section until the final result proof.

\subsection{Asymptotic Result for Primitive Unital Channels}

We first show how Watrous' Theorem~\ref{watrous} yields an asymptotic result for the case of primitive unital channels. We begin with a result that we expect is well-known (see \cite{carlen} for instance), but for completeness we provide a short proof.

\begin{lemma}\label{cond-exp-mu}
Let $\mathcal{A}$ be any unital $*$-subalgebra of $M_d$. Let $\mathcal{E}_{\mathcal{A}}$ be the trace preserving conditional expectation onto $\mathcal{A}$. Then $\mathcal{E}_{\mathcal{A}}$ is a mixed unitary channel.
\end{lemma}

\begin{proof}
Let $\mathcal{U}(\mathcal{A}')$ be the unitary group of the commutant algebra $\mathcal{A}'=\{X\in M_d: A X=X A, \forall A\in \mathcal{A}\}.$
It follows that the conditional expectation can be written as follows for all $X\in M_d$:
\[ \mathcal{E}_{\mathcal{A}}(X)=\int_{U \in \mathcal{U}(\mathcal{A}')} UXU^* d\mu(U) ,\]
where $\mu(U)$ is the normalized Haar measure on $\mathcal{U}(\mathcal{A}')$.
Indeed, it is easy to see that this integral operator is trace preserving, projects onto $\mathcal{A}$, and satisfies the other conditional expectation properties from the invariance of the Haar measure. So by uniqueness the map is $\mathcal E_{\mathcal A}$.

Now as the commutant  $\mathcal{A}'$ is a finite dimensional subalgebra, the group $\mathcal{U}(\mathcal{A}')$ is closed and hence the convex hull of the set $\{UXU^*: U \in \mathcal{U}(\mathcal{A}') \}$ is a closed convex set. Thus $\mathcal{E}_{\mathcal{A}}(X)$ lies in the convex hull of mappings of the form $ UXU^*$, with $U \in \mathcal{U}(\mathcal{A}')$, and so
$\mathcal{E}_{\mathcal{A}}$ is a mixed unitary map.
\end{proof}

\begin{remark}
Note that from the above result, it is clear that the completely depolarizing channel $\delta_d(X) = d^{-1}\tr (X) I_d $, which is the trace preserving conditional expectation onto the trivial algebra $\mathcal A = \{\mathbb{C} I\}$, is a mixed unitary map. A concrete representation of this map can be written down:
\[\delta_d(X) = d^{-1}\tr (X) I_d = \frac{1}{d^2} \sum_{a,b=0}^{d-1} W_{a,b} X W_{a, b}^*, \]
where $W_{a,b}$ are the Weyl-Heisenberg unitaries defined by
\[W_{a,b}= S^a D^b, 0\leq a, b\leq d-1, \]
and $S=\sum_{j=1}^d E_{j+1, j}\in M_d$ is the forward cyclic shift operator and $D= \sum_{j=1}^d \omega^j E_{j,j}\in M_d$ is the `clock operator' with $\omega=\text{exp}(2\pi i/d)$.
\end{remark}

We use the above result to prove the following. Let us denote $\text{MU(d)}$ to be the set of all mixed-unitary channels on $M_d$, which note is a closed convex set of linear maps. See \cite{paulsen} for basic properties of the completely bounded distance measure.

\begin{lemma}\label{primitivelemma}
Let $\Phi:M_d\rightarrow M_d$ be any unital quantum channel. Then
\[ \displaystyle {{\liminf} _{n\to\infty} } \, d_{CB}( \Phi^n, \text{MU(d)}) =0,\]
where $d_{CB}( \Phi^n, \text{MU(d)})$ is the completely bounded distance of $\Phi^n$ from the closed convex set $\text{MU(d)}$.
\end{lemma}

\begin{proof}
For the unital channel $\Phi$, look at the semigroup of linear maps $\mathfrak{C}_\Phi=\{\Phi^n: n\in \mathbb{N}\}$. As the closed unit ball of linear maps from $M_d$ to $M_d$ is Bolzano-Weierstrass compact, the above semigroup admits at least one limit point. By Kuperberg's Theorem (see \cite{kuperberg}) there is a subsequence $n_1, n_2, \cdots, $ such that
\[\lim_{j\to\infty} \Phi^{n_j} =\mathcal{E}_\Phi,\]
where $\mathcal{E}_\Phi$ is the conditional expectation channel onto the algebra generated by the eigen-operators of $\Phi$ corresponding to the eigenvalues $\lambda$ with $|\lambda|=1$ (this is the peripheral algebra $\mathcal{M}_{\Phi^\infty}$ studied in \cite{miza-mult}).  Now the result follows from Proposition \ref{cond-exp-mu}.
\end{proof}

For a special class of channels (e.g., see \cite{ahiable2021entanglement,rahaman2019new,sanz2010quantum}), one can make a stronger statement.

\begin{definition}
A unital channel $\Phi:M_d\rightarrow M_d$ is {\it primitive}  if it is {\it irreducible} (i.e., $\Phi(P)\leq \lambda P$ for some projection $P$ implies $P=0$ or $P=I$) and it has a trivial {\it peripheral spectrum} (i.e.,  $\text{spec}(\Phi)\cap \mathbb{T}= \{1 \}$).
\end{definition}

\begin{theorem}\label{primitivetheorem}
For every primitive unital channel $\Phi:M_d\rightarrow M_d$, there is a finite $k\in \mathbb{N}$ such that $\Phi^k$ is mixed unitary, and subsequenctly for every $l\geq k$, $\Phi^l$ is mixed unitary.
\end{theorem}

\begin{proof}
It follows from the proof of the previous result that for a primitive unital channel $\Phi$, the conditional expectation $\mathcal{E}_\Phi$ described above is the completely depolarizing channel $\delta_d(X) $. Now by Watrous's theorem (\ref{watrous}) there is a ball around $\delta_d(X) $ where every unital channel is mixed-unitary.
So from the subsequence  $n_1, n_2, \cdots, $  if we take sufficiently large $n_i$'s, the maps $\Phi^{n_i}$ must fall in the ball around $\delta_d(X) $ . Hence there is a $k\in \mathbb{N}$ such that $\Phi^k$ is in this ball and it is mixed unitary.

The second statement follows easily from the above argument and the CB norm estimate: \[||\Phi^{k+r}-\delta_d ||_{CB}=||\Phi^r (\Phi^k-\delta_d)||_{CB}\leq || (\Phi^k-\delta_d)||_{CB} , \]
which also uses the fact that $\Phi \circ \delta_d = \delta_d$ as $\Phi$ is unital.
\end{proof}

In what follows, we will show how Theorem~\ref{main} allows us to prove an analogous result for all unital channels.

\subsection{Irreducible Channels and Peripheral Eigenvalues}

Next we derive some properties of peripheral eigenvalues for irreducible unital channels.

Let us first observe that a unital channel $\Phi$ is irreducible if and only if its fixed point algebra $\mathrm{Fix}(\Phi)$ is just the scalar algebra, $\mathcal A = \C I$. Indeed, if $\Phi$ is irreducible, then only the trivial projections are fixed by $\Phi$, and hence its fixed point algebra (which is spanned by its projections as a von Neumann algebra) must be trivial. Conversely, if the fixed point algebra for $\Phi$ is trivial and $P$ is a projection with $\Phi(P) \leq \lambda P$, then $\Phi(P)$ is supported on the range of $P$, which for a unital channel implies (as proved in \cite{kribscommut}) that in fact $\Phi(P)=P$ if it is non-zero, and hence $P=0$ or $P=I$.

In the following result, we denote the set of unital channels that fix a given algebra $\mathcal A$ by $\mathcal F(\mathcal A)$. Evidently this set has the structure of a convex semigroup under composition of maps. It is also $\ast$-closed, in the sense that a map is in the set if and only if its dual map is as well (which can be seen as a consequence of the fixed point theorem for unital channels \cite{kribscommut}).

\begin{lemma}\label{stariso}
Let $\mathcal{A}$ be a unital subalgebra of $M_d$ that is unitarily equivalent to $\oplus_{k=1}^r I_{m_k} \otimes M_{n_k}$, and let $\mathcal{F}(\mathcal{A})$ be the semigroup of unital channels on $M_d$ that fix $\mathcal{A}$. Let $\widehat{\mathcal{A}} = \oplus_{k=1}^r \mathbb{C} I_{m_k}$, with associated semigroup $\mathcal{F}(\widehat{\mathcal{A}})$ of unital channels on $M_{(\sum_k m_k)}$ that fix $\widehat{\mathcal{A}}$.

Then, there is a convex $\ast$-semigroup isomorphism $\alpha : \mathcal{F}(\mathcal{A}) \rightarrow \mathcal{F}(\widehat{\mathcal{A}})$ with the property that $\alpha(\Phi) \in \mathcal{F}(\widehat{\mathcal{A}})$ is mixed unitary if and only if $\Phi \in \mathcal{F}(\mathcal{A})$ is mixed unitary.
\end{lemma}

\begin{proof}
Let $\Phi \in \mathcal{F}(\mathcal{A})$ with Kraus operators $\{K_i\}_{i=1}^n$. Since $\Phi$ fixes $\mathcal{A}$, we have $\mathcal A \subseteq \mathrm{Fix}(\Phi) = \{ K_i \}^\prime$, and so $K_i \in \mathcal{A}'$. Hence there exists a unitary $U\in M_d$ and matrices $K_{ik}\in M_{m_k}$ such that $U^*K_i U = \oplus_{k=1}^r K_{ik}\otimes I_{n_k}$ for all $i$.

Then $\alpha(\Phi)$ is defined to be the map whose Kraus operators are $\widehat{K_i}:=\oplus_{k=1}^r K_{ik}$, which, as a notational convenience, we sometimes write as $\widehat{K_i} = \alpha(K_i)$. To show that $\alpha(\Phi)$ fixes $\widehat{\mathcal{A}}$, we note that the Kraus operators of $\alpha(\Phi)$ always lies in the algebra $\oplus_{k=1}^r M_{m_k}$, and that $\widehat{\mathcal{A}}$ thus necessarily commutes with $\widehat{K_i}$; hence it is contained in the fixed point algebra.

It is easy to see that the image of $\alpha$ does not depend on a particular operator-sum representation of $\Phi$, and, moreover, that an operator-sum representation of $\Phi$ is minimal in terms of number of Kraus operators if and only if the image of the Kraus operators under $\alpha$ is a minimal representation of $\alpha(\Phi)$.

It is also clear that $\alpha$ is a $*$-homomorphism, since for any $\Phi$, $\Psi \in \mathcal{A}$ with respective Kraus operators $U^*K_i U = \oplus_{k=1}^r K_{ik}\otimes I_{n_k}$ and $U^*L_i U = \oplus_{k=1}^r L_{ik}\otimes I_{n_k}$, we have that $\Phi\circ \Psi$ has Kraus operators $U\bigl(\oplus_{k=1}^r K_{ik}L_{i'k}\otimes I_{n_k}\bigr)U^*$. The image under $\alpha$ of these operators are $\oplus_{k=1}^r K_{ik}L_{i'k}$, which are exactly the Kraus operators of $\alpha(\Phi)\circ \alpha(\Psi)$.

To see that $\alpha$ is an isomorphism, let $\Phi,\Psi \in \mathcal F(\mathcal A)$, and suppose $\Phi$ has a minimal set of Kraus operators given by $\{K_i = U\bigl( \oplus_{k=1}^r K_{ik}\otimes I_{n_k}\bigr)U^*\}_{i=1}^n$ and $\Psi$ has a minimal set of Kraus operators $\{L_i = U\bigl(\oplus_{k=1}^r L_{ik}\otimes I_{n_k}\bigr)U^*\}_{i=1}^{n'}.$ Then $\alpha(\Phi)$ has Kraus operators $\{\widehat{K_i}= \oplus_{k=1}^r K_{ik}\}_{i=1}^n$ and $\alpha(\Psi)$ has $\{\widehat{L_i} = \oplus_{k=1}^r L_{ik}\}_{i=1}^{n'}$.
If $\alpha(\Phi) = \alpha(\Psi)$, then we have $\{\widehat{K_i}\}_{i=1}^n$ and $\{\widehat{L_i}\}_{i=1}^{n'}$ are two different minimal Kraus representations of the same channel, so $n=n'$. Hence there exists a scalar unitary matrix $V = (v_{ij})$ such that $\widehat{L_i} = \sum_j v_{ij}\widehat{K_j}$ and so
$L_{ik} = \sum_j v_{ij}K_{jk}.$
Thus
\begin{eqnarray*}
L_i  = U\bigl(\oplus_{k=1}^r L_{ik}\otimes I_{n_k}\bigr)U^*
&=& U\bigl(\oplus_{k=1}^r \bigl(\sum_{j=1}^n v_{ij}K_{jk}\bigr)\otimes I_{n_k}\bigr)U^*  \\
&=& U\bigl(\sum_j v_{ij} \bigl(\oplus_{k=1}^p K_{jk}\otimes I_{n_k}\bigr)\bigr)U^*
= \sum_j v_{ij}K_j ,
\end{eqnarray*}
and so $\{L_i\}$ and $\{K_i\}$ are two different representations of the same channel, giving $\Phi = \Psi$.

Finally, $\alpha(\Phi)$ is a unitary adjunction channel if and only if $\Phi$ is; since $\oplus_{k=1}^r U_k \otimes I_{j_k}$ is unitary if and only if each $U_k$ is unitary, which in turn is equivalent to $\oplus_{k=1}^r U_k$ being unitary.
So, in one direction, if $\Phi = \sum_i p_i \mathrm{ad}_{U_i}$ expresses $\Phi$ as a convex combination of unitary adjunction maps (where $\mathrm{ad}_U(X) = UXU^*$), the (convex) linearity of $\alpha$ guarantees that
\[
\alpha(\Phi) = \sum_i p_i \alpha(\mathrm{ad}_{U_i})
\]
expresses $\alpha(\Phi_i)$ as a convex combination of the unitary adjunctions $\alpha(\mathrm{ad}_{U_i})$.
In the other direction, suppose
\[
\alpha(\Phi) = \sum_i p_i \mathrm{ad}_{\widehat{U_i}}
\]
for some unitaries $\widehat{U_i}$. Since $\alpha(\Phi)$ fixes the algebra $\widehat{\mathcal{A}}=\oplus_{k=1}^r I_{m_k}$, it must be that each $\widehat{U_i} \in \widehat{\mathcal{A}}' = \oplus_{k=1}^r M_{m_k}$ and hence $\widehat{U_i} = \oplus_{k=1}^r U_{ik}$ for some unitaries $U_{ik}$ on each block.
If we define $U_i = U\bigl(\oplus_{k=1}^r U_{ik}\otimes I_{n_k}\bigr)U^* \in \mathcal{A}$, it is clear that $\alpha(\Phi)$ is now the image of $\Psi:=\sum_i p_i \mathrm{ad}_{U_i}$, which is mixed unitary. Since $\alpha$ is an isomorphism, and $\alpha(\Phi) = \alpha(\Psi)$, it must in fact be that $\Phi = \Psi$ and hence is mixed unitary.
\end{proof}

\begin{remark}
{\rm 
Notice that if $\mathcal{A}$ is the fixed point algebra of $\Phi$, i.e., the largest unital algebra fixed by $\Phi$, then the algebra generated by its Kraus operators $K_i$ is $\mathcal{A}'$. So $\alpha(\Phi)$ has Kraus operators that generate the algebra $\alpha(\mathcal{A}) = \oplus_{k=1}^r M_{m_k}$, and so the fixed point algebra of $\alpha(\Phi)$ is the abelian algebra $\oplus_{k=1}^r \C I_{m_k}$. Also notice that the channels $\Phi_k$, with Kraus operators $\{K_{ik}\}$ are irreducible.
Thus, without loss of generality, we will prove our result for channels with abelian fixed point algebra, as any unital channel is identified with a channel that has abelian fixed point algebra, and where the identification carries through the relevant properties (i.e., commutes with powers and preserves mixed unitarity).
}
\end{remark}

We next consider the {\it peripheral spectrum} for a map $\Phi : M_d \rightarrow M_d$, which is the set
\[
\{ X\in M_d \,\, | \,\, \Phi(X) = \lambda X \,\, \mathrm{for\, some} \, |\lambda|=1\}.
\]
In the case of an irreducible unital channel, there is a positive integer $m$ such that the peripheral spectrum is $\{\omega^i\}_{i=0}^{m-1}$ for some primitive $m^{th}$ root of unity (see for instance Theorem~6.6 from \cite{wolf2012quantum}). Further, as shown in \cite{miza-mult} (Theorem 2.5), for a unital channel $\Phi$, the algebra generated by all peripheral eigen-operators for $\Phi$ is equal to the algebra $\mathcal M_{\Phi^\infty}$, which is defined as the decreasing intersection of the multiplicative domains $\mathcal M_{\Phi^k}$ for $\Phi^k$, $k\geq 1$; in particular, the peripheral spectrum of $\Phi^k$ is a subset of the peripheral spectrum of $\Phi$.

The following useful fact for us comes as a simple consequence of the spectral mapping theorem, from which it follows that the spectrum of $\Phi^m$ consists of the elements of the spectrum of $\Phi$ raised to the $m$th power.

\begin{lemma}\label{peripherallemma}
Suppose $\Phi$ is an irreducible unital channel with  peripheral spectrum $\{\omega^i\}_{i=0}^{m-1}$ for some primitive $m^{th}$ root of unity. Then $\Phi^m$ has no non-trivial peripheral spectrum; that is, $\mathrm{spec}(\Phi^m)\cap \mathbb T = \{ 1 \}$.
\end{lemma}


We next recall basic features of peripheral eigenvalues, with a short proof for completeness.

\begin{lemma}\label{peripheigs} Let $\Phi : M_d \rightarrow M_d$ be a unital channel, and let $X$ be a peripheral eigenvector for $\Phi$: $\Phi(X) = \lambda X$ for some $|\lambda|=1$. Then $K_i X = \lambda XK_i$ for all Kraus operators $K_i$, and so if $X$ is a peripheral eigenvector for $\Phi$ with eigenvalue $\lambda$, then we have
\[
\Phi(XA) = \lambda X \Phi(A) \ \, \Phi(AX) = \overline{\lambda}\Phi(A)X ,
\]
for all $A\in M_d$.
\end{lemma}

\begin{proof}
Define $A_i = K_i X - \lambda XK_i$. Then we have,
\begin{align*}\sum_i A_i A_i^* & = \sum_i K_i X X^* K_i^* -\overline{\lambda}\sum_i K_i XK_i^*X^* - \lambda X \sum_i K_i X^*K_i^* + |\lambda|^2 X \sum_i K_iK_i^* X^* \\
& = \Phi(XX^*) - XX^*,
\end{align*}
where we use the fact that $\Phi(X) = \lambda X$, $\Phi(X^*) = \overline{\lambda} X^*$, and that $\Phi$ is unital.
Then, by trace preservation, we have that
\[
\sum_i \tr(A_iA_i^*)  = \tr(\Phi(XX^*)) - \tr(XX^*)  = 0 ,
\]
and hence each $A_i = 0$. The final statement immediately follows.
\end{proof}

We also need the following characterization of peripheral eigenvectors in the commutative fixed point algebra case.

\begin{lemma}\label{peripheraldecomp}
Let $\Phi$ be a unital channel with fixed point algebra unitarily equivalent to $\oplus_{k=1}^r \C I_{m_k}$. Let $X$ be a peripheral eigenvector. Then one of the two following cases holds:
\begin{enumerate}
\item  $X = \oplus_{k=1}^r X_k$ where each $X_k$ is a peripheral eigenvector for the irreducible channel $\Phi_k$ obtained by restricting the Kraus operators of $\Phi$ to the $k^{th}$ diagonal block.
\item There exists $j,k$ such that $m_j = m_k$, and there is a unitary $U$ on $\C^{m_j}$ such that $\Phi_j = \mathrm{ad}_U \circ \Phi_k \circ \mathrm{ad}_{U^*}$.
\end{enumerate}
\end{lemma}

\begin{proof}
Up to unitary equivalence, the fixed point algebra has minimal central (orthogonal) projections
$P_i = \oplus_{k=1}^r \delta_{ik} I_{i_k}$ with $\sum_i P_i = I$.
As these are fixed points of $\Phi$, we have by Lemma~\ref{peripheigs} that
$\Phi(P_kXP_j) = P_k \Phi(X) P_j$ for all $X$ and $k,j$. In particular, applying this to the peripheral eigenvector $X$ with eigenvalue $\lambda$, we get
\[
\Phi(P_kX P_j) = \lambda P_kX P_j
\]
for all pairs $j,k$. That is, $P_kX P_j$ is also a peripheral eigenvector for $\Phi$ with eigenvalue $\lambda$.

Now, Lemma~\ref{peripheigs} also shows that, for any peripheral eigenvector $X$ we have
\[
\Phi(XX^*) =  |\lambda|^2 XX^* = XX^*,
\]
and so $XX^*$ must in the span of the $P_i$. Hence the same is true for $X^*X$, $P_k X P_j X^* P_k$ and $P_k X^* P_j X P_k$.

Thus we can find scalars $c_i$ such that
\[
P_kXP_jX^*P_k = \sum_i c_i P_i ,
\]
which yields after multiplying on the left or right by $P_k$ that
\[
P_k XP_jX^*P_k =  c_k P_k.
\]
If we let $X_{kj}$ be the operator corresponding to the $(k,j)$ block in the decomposition determined by the $\{P_j\}$, which is $P_k X$ restricted to the range of $P_j$, then we have that
\[
X_{kj}X^*_{kj} = c_k I_{m_k} \quad \textnormal{and} \quad X_{kj}^*X_{kj} = c_j I_{m_j}.
\]

There are two possibilites: either $c_k = c_j = 0$, or both scalars are non-zero and $X_{kj}$ is a (non-zero) multiple of a unitary (and $m_j = m_k$).
Thus, in this block matrix form, any peripheral eigenvector has the form,
\[
X = \sum_{i,j} E_{ij} \otimes X_{ij},
\]
where each $X_{ij}$ is either $0$ or a (non-zero) multiple of a unitary. Moreover, we know by Lemma~\ref{peripheigs} that $K_i X = \lambda X K_i$, and so we have that
\[
K_{ij}X_{jk} = \lambda X_{jk} K_{ik}
\]
for all $i$ and all $(j,k)$. In the case that $X_{jk}$ is non-zero, we therefore have,
\[
K_{ij} = \frac{1}{c_j}\lambda X_{jk}K_{ik}X_{jk}^*.
\]
Since $\frac{X_{jk}}{\sqrt{c_j}}$ is unitary, and $|\lambda| = 1$, this expresses $K_{ij}$ as a unitary conjugation of $K_{ik}$ for all $i$; that is, if $\Phi_i$ is the channel whose Kraus operators are $\{K_{ij}\}_{j=1}^n$, then $\Phi_j = \mathrm{ad}_U \circ \Phi_k \circ \mathrm{ad}_{U^*}$  with $U = \sqrt{\frac{\lambda}{c_j}}X_{jk}$.
\end{proof}

Combining these last results with Kuperberg's Theorem \cite{kuperberg} and our main result from the last section, allows us to prove the following.

\begin{theorem}\label{eventually-mu}
Let $\Phi$ be a unital channel. Then there exists an integer $k>0$ such that $\Phi^k$ is mixed unitary.
\end{theorem}

\begin{proof}
By Lemma~\ref{stariso}, perhaps by replacing $\Phi$ with $\alpha(\Phi)$, we can without loss of generality assume $\Phi$ has a commutative fixed point algebra. Then, Lemma~\ref{peripherallemma} and Lemma~\ref{peripheraldecomp} show that a high enough power, $M\geq 1$ say, of $\Phi$ has no non-trivial peripheral spectrum; for instance, $M$ can be taken as the lowest common multiple of the $m$'s from Lemma~\ref{peripherallemma} applied to a representative from each of the irreducible channel unitary equivalence classes found in Lemma~\ref{peripheraldecomp}.

Thus, $\Phi^M$ is a unital channel with no non-trivial peripheral spectrum, and so its peripheral algebra is just its fixed point algebra. We can now apply Kuperberg's Theorem in this case to find a subsequence $\{k_i\}$ such that $(\Phi^M)^{k_i} = \Phi^{M k_i} \rightarrow \mathcal E_{\mathcal{A}}$ where $\mathcal{A}$ is the fixed point algebra of $\Phi^M$. By Theorem~\ref{main}, there is a ball around $\mathcal E_{\mathcal{A}}$ consisting entirely of mixed unitaries, and hence, any channel in $\mathcal{F}(\mathcal{A})$ sufficiently close to $\mathcal E_{\mathcal{A}}$ is mixed unitary. Therefore, it follows that there is a $k_N$ such that, for all $i>N$, the channel $(\Phi^M)^{k_j}$ is sufficiently close to $\mathcal E_{\mathcal{A}}$ that it is mixed unitary, and this completes the proof.
\end{proof}

\begin{remark}
{\rm 
Notice that in order to obtain this result, we cannot use Kuperberg's Theorem directly with the conditional expectation onto the peripheral algebra; this is because the ball of mixed unitaries we obtain around $\mathcal E_{\mathcal{A}}$ is in the relative interior of $\mathcal{F}(\mathcal{A})$, the set of all unital channels with fixed point $\mathcal{A}$. So if $\Phi^M$ only has peripheral algebra $\mathcal{A}$, but not fixed point algebra $\mathcal{A}$, although $\Phi^{M k_i} \rightarrow \mathcal E_{\mathcal{A}}$, it may approach from outside the relative interior $\mathcal{F}(\mathcal{A})$ where the Theorem does not apply.
}
\end{remark}

\begin{remark}
{\rm 
We also draw the attention of the reader to a conjecture called the ``Asymptotic Quantum Birkholff Conjecture'', which asks whether for a unital quantum channel $\Phi: M_n \rightarrow M_n$, it holds that,
\[
\lim_{k \rightarrow \infty} d_{CB} (\Phi^{\otimes k}, MU(M_n^{\otimes k})) = 0,
\]
where, as above, $d_{CB}$ is the completely bounded distance of $\Phi^{\otimes k}$ to the set of mixed unitary maps on $M_n^{\otimes k}$. The conjecture was resolved in the negative by Haagerup and Musat (\cite{HaagerupM2011}). They introduced a new class of maps called \textit{factorizable maps} and showed that maps which are not factorizable, fail to satisfy the above conjecture. In essence, this means that not every unital channel, after taking tensor powers with itself, becomes mixed unitary even if we take larger and larger tensor powers.  In contrast, Lemma \ref{primitivelemma} shows that every unital channel `asymptotically becomes' mixed unitary. Quite significantly, Theorem \ref{eventually-mu} goes further and uncovers an interesting aspect of unital channels the contrasts with tensor powers: it says under composition, every unital channel becomes mixed unitary after finitely many applications.
}
\end{remark}


\section{The Case of the Diagonal Algebra: Correlation Matrices and Schur Product Channels}

We finish by considering the case of the diagonal algebra in Theorem~\ref{main} in more detail; that is, $\mathcal{A} = \Delta_d \cong \oplus_{k=1}^d \C 1$, the algebra of $d \times d$ diagonal complex matrices. We shall give two alternate proofs of the theorem in this case using different approaches, and in doing so, we make connections with the theory of correlation matrices and Schur product maps \cite{paulsen,holevo2019quantum} (which have also recently arisen in other quantum information settings \cite{harris2018schur,levick2017quantum,puchala2021dephasing}), and Abelian group theory.

We begin by noting that the trace preserving conditional expectation onto $\Delta_d$ is the map-to-diagonal, defined by
\begin{equation}\Delta(X) = \sum_{i=1}^d x_{ii} E_{ii}, \end{equation}
where $X = (x_{ij})$ and $E_{ij}$, $1 \leq i,j \leq d$, are the matrix units for $M_d$.

Recall that a correlation matrix is a positive semi-definite matrix with 1's down its main diagonal. Further, the Schur (or Hadamard) product of two matrices $A = (a_{ij}), B = (b_{ij})\in M_d$ is $A\circ B = (a_{ij}b_{ij})$. Given any $C\in M_d$, one can define a linear map $\Phi(X) = X \circ C$, and then $\Phi$ is completely positive if and only if $C$ is a positive semidefinite matrix \cite{paulsen}. It is also clear that such a map is unital if and only if it is trace preserving.

\begin{proposition}\label{schurprop} \cite{kye, LW}
Any unital channel $\Phi : M_d \rightarrow M_d$ whose fixed point algebra contains $\Delta_d$ is a Schur product channel; that is, there exists a correlation matrix $C$ such that
\[\Phi(X) = X\circ C, \] where $\circ$ denotes the Schur product.
\end{proposition}



Since the commutant of $\mathcal A = \Delta_d$ is $\mathcal A^\prime = \Delta_d$ again, as a consequence of Proposition~\ref{schurprop}, we have Theorem~\ref{main} restated in this particular case as follows.

\begin{theorem}\label{diagvers}
There exists a constant $0 \leq p \leq 1$ such that for all Schur product channels $\Phi : M_d \rightarrow M_d$, the map
\[
p \Phi + (1-p)\Delta
\]
is a mixed unitary channel defined by diagonal unitary matrices.
\end{theorem}

We provide the following alternative proof for this case.

\begin{lemma}
A channel $\Phi : M_d \rightarrow M_d$ is of the form $UXU^*$ where $U$ is a diagonal unitary if and only if $\Phi(X) = X \circ C$ where $C = z z^*$ is a rank-one correlation matrix with $z\in \C^d$ and $|z_i| = 1$ for all $i$.
\end{lemma}

\begin{proof}
If $U = \mathrm{diag}(z_1,\cdots, z_d)$ is unitary, then $|z_i| = 1$ and
\[UXU^* = \sum_{i,j} z_iz_j^*x_{ij} E_{ij} = X\circ C , \] where $c_{ij} = z_iz_j^*$.

Conversely, if $C$ is a rank-one correlation matrix, then $C = zz^*$ for some vector $z = (z_1,\cdots, z_d)^T$. We have $c_{ii} = 1 = z_iz_i^* = |z_i|^2$, and now it is easy to see that $X\circ C$ is equal to $UXU^*$ where $U = \mathrm{diag}(z)$ is unitary since each $|z_i| = 1$.
\end{proof}

Since the map-to-diagonal $\Delta$ is equal to the Schur-product channel with the correlation matrix $I_d$, Theorem \ref{diagvers} can be restated as follows. This is the version that we prove; equivalence to the previously stated version follows by replacing all Schur product maps with their associated correlation matrices or vice-versa.

\begin{theorem}\label{Haar}
There exists a constant $0 \leq p \leq 1$ such that every $d\times d$ correlation matrix $C$ satisfies that
\[
p C + (1-p)I_d
\]
is in the convex hull of rank-one correlation matrices.
\end{theorem}

\begin{proof} Let $C$ be a correlation matrix. Let $z = (z_1,\cdots, z_d)^T$ where $|z_i| = 1$, and take the integral
\begin{equation}
\int_{z_1,\cdots, z_d} zz^* \langle z,Cz\rangle d\mu(z_1)\cdots d\mu(z_d) ,
\end{equation}
where the measure is just Haar measure on the unit circle. As $\langle z, Cz\rangle = \sum_{k,l=1}^d c_{kl}z_k^*z_l$, we can write this as
\[
\sum_{k,l}c_{kl} \sum_{i,j}E_{ij} \int_z z_iz_j^*z_k^*z_l d\mu(z) ,
\]
and since $\int z_i^k d\mu(z_i) = 0$ for any $k\neq 0$, the only non-zero terms in this sum come when either $i =j$ and $k=l$ or $i=k$ and $j=l$, or the intersection, $i=j=k=l$. Thus, to avoid double-counting, we get the following result:
\[
\sum_{k,l}c_{kl} E_{kl} + \sum_{i,k} c_{kk}E_{ii} - \sum_i c_{ii}E_{ii}
\]
which, since $c_{ii} = 1$, is just
\[
C + (d-1)I_d.
\]
After suitably normalizing, we see that the integral gives $\frac{1}{d}(C + (d-1)I_d)$. Since $zz^*$ is always a rank-one correlation matrix, and $\langle z,Cz\rangle$ is always positive, we have written this correlation matrix as a positive combination of rank-ones; normalizing makes it a convex combination, proving the result, with $p = \frac{1}{d}$.
\end{proof}

\begin{remark}
{\rm 
We mention here that the above result elucidates the fact that the identity matrix is in the interior of the set of all correlation matrices that can be written as a convex combination of rank-1 correlation matrices. This fact was previously pointed out in the article \cite{dykema-juschenko} (cf. section 4). Here we have found a new way to realize this fact and our method evidently provides better estimates of the convex combinations in some cases, based on a cursory comparison to the estimates of \cite{dykema-juschenko}.
}
\end{remark}


\subsection{Group Theory Approach}


Let $G$ be an Abelian group. We let $\widehat{G}$ be the set of all group homomorphisms from $G$ to $\mathbb{T}$, the unit circle in the complex plane.  The set $\widehat{G}$ is a group under multiplication and is called the dual group.  The Abelian groups $\mathbb{Z}^d$ and $\mathbb{T}^d$ are duals to one another and any finite Abelian group is self-dual.  Let $\mu$ be any measure on an Abelian group $G$, then the Fourier transform of $\mu$ is the complex valued function on $\widehat{G}$ defined as follows: $\widehat{\mu}(\chi)=\int_G \chi(g) d\mu (g)$.  A complex-valued function on $\widehat{G}$ is said to be positive definite if it is the Fourier transform of a measure on $G$.
Reminiscent of the standard basis in linear algebra, if our group $G^d$ is either $G=\mathbb{Z}_m$ or $G=\mathbb{Z}$, then $e_k$ denotes the element in $G^d$ consisting of an $n$-tuple of elements of $G$ where the $k$th element is $1$ and all other elements are $0$.

We can characterize the convex hulls of rank one correlation matrices in both the real and the complex cases in terms of positive definite functions. The real version of the result which we present first is essentially equivalent to \cite[Proposition 2.1]{guptpar} and \cite[Theorem 7]{o2013self}.

\begin{theorem} \label{BGP} Let $C$ be an $d \times d$ real matrix.  Then $C$ is in the convex hull of the real rank one correlation matrices if and only if there exists a positive definite function $f:\mathbb{Z}_2^d\to \mathbb{R}$ with the following properties:
\begin{enumerate}
\item $f(0)=1$
\item $f(e_i-e_j)=c_{ij}$ for $1\le i<j \le d$
\end{enumerate} \end{theorem}

The complex version of this theorem, which we now state, appears to be new.

\begin{theorem}\label{comex} Let $C$ be an $d \times d$ complex matrix.  Then $C$ is in the convex hull of the complex rank one correlation matrices if and only if there exists a positive definite function $f:\mathbb{Z}^d\to \mathbb{C}$ with the following properties:
\begin{enumerate}
\item $f(0)=1$
\item $f(e_i-e_j)=c_{ij}$ for $1\le i<j \le d$
\end{enumerate} \end{theorem}

We can combine these two versions into a common generalization as follows.

\begin{theorem} \label{BGP3} Let $C$ be an $d \times d$ complex matrix. Let $G$ be any topologically closed subgroup of $\mathbb{T}$.   Then $C$ is in the convex hull of the rank one correlation matrices with all entries in $G$ if and only if there exists a positive definite function $f:\widehat{G}^d\to \mathbb{C}$ with the following properties:
\begin{enumerate}
\item $f(0)=1$
\item $f(e_i-e_j)=c_{ij}$ for $1\le i<j \le d$
\end{enumerate} \end{theorem}

Setting $G=\mathbb{Z}_2$ in Theorem \ref{BGP3} gives us Theorem \ref{BGP} and setting $G=\mathbb{T}$ in Theorem \ref{BGP3} gives us Theorem \ref{comex}.  Hence we only need prove Theorem \ref{BGP3}.

\begin{proof} Let $G$ be any topologically closed subgroup of $\mathbb{T}$.  If $v$ is any $n$-vector all of whose entries are in $G$, then let $\delta_v$ denote the probability measure on $G^d$ satisfying $\delta_v(\{v\})=1$. Let $\widehat{\delta}_v$ be the corresponding positive definite function (i.e. for any $\chi \in \widehat{G}^d$, $\widehat{\delta}_v(\chi)=\int_G \chi(g) d\delta_v=\chi(v))$.  If $\chi=(c_1,c_2,...,c_d)\in \widehat{G}^d$ and $v=(v_1,v_2,...,v_d)\in G^d$, then $\chi(v)=\prod_{k=1}^dv_k^{c_k}$.  Hence $\widehat{\delta}_v(e_i-e_j)=v_iv_j^{-1}=v_i\overline{v_j}$ since $|v_j|=1$.  Therefore for all $i,j$, $\widehat{\delta}_v(e_i-e_j)$ is the $(i,j)$th entry of the matrix $vv^*$.  Now if $C$ is in the convex hull of the rank one correlation matrices with all entries in $G$, there exists $\{ \lambda_i\}_i$ positive numbers summing to one and $\{ v_i\}_i$ $n$-vectors having all elements in $G$ such that $C=\sum_i\lambda_iv_iv_i^*$.  It follows from linearity that the Fourier transform of the probability measure $\sum_i \lambda_i\delta_{v_i}$ is the $f$ which satisfies all the hypotheses of the theorem.  The converse follows by reversing the steps of this argument. For the $G=\mathbb{T}$ case, we note that the set of all probability measures on $\mathbb{T}$ is weak-$*$ compact by the Banach-Alaoglu theorem and hence is the closed convex hull of the point measures on $\mathbb{T}$ by the Krein-Milman theorem. The result now follows using a similar argument to the topologically closed subgroup case.
\end{proof}


 We can use Theorem \ref{comex} to construct an improvement on Theorem \ref{Haar}.  We begin with the following lemma which gives a useful example of a positive definite function on the integers.

 \begin{lemma} Let $c$ be a complex number of modulus less than or equal to one. Then the function  $f_c(n):\mathbb{Z}\to \mathbb{C}$ defined as
 \[f_c(n)= \begin{cases}
      c^n & n\geq 0 \\
      \overline{c}^n & n<0
   \end{cases}
\]
is positive definite.
\end{lemma}

\begin{proof} Note that the Mobius transformation $g(z)=\frac{1}{1-z}$ maps the closed unit disk of the complex plane to the half plane $\{z: Re(z)\ge \frac{1}{2}\}$.  It follows from this that when $\vert c\vert\le 1$ and $c\neq 1$, then $\widehat{f_c}(e^{i\theta})=\sum_{k\in \mathbb{Z}}f_c(k)e^{ik\theta}=-1+\frac{1}{1-ce^{i\theta}}+ \frac{1}{1-ce^{-i\theta}}\geq 0$.  Hence $f_c$ is positive definite.  The function $f_1$ is the Fourier transform of the point measure at zero and hence is positive definite. \end{proof}

This has some important implications for correlation matrices.

\begin{corollary} \label{corold} Let $v=(v_1,v_2,..,v_d)\in \mathbb{C}^d$ with $\Vert v\Vert_{\infty}\le 1$ and let $M(v)$ denote the $d \times d$ matrix having the same off-diagonal entries as $vv^*$ and all diagonal entries equal to one.  Then $M(v)$ is in the convex hull of the complex rank one correlation matrices. \end{corollary}

\begin{proof} It follows from the previous lemma that $f_{v_k}$ is a positive definite function on $\mathbb{Z}$.  Therefore a simple product measure argument shows us that $f(n_1,n_2,...,n_d)=\prod_{k=1}^d f_{v_k}(n_k)$ is a positive definite function $f:\mathbb{Z}^d\to \mathbb{C}$.  Then $f(0)=1$ and $f(e_i-e_j)=  f_{v_i}(1)f_{v_j}(-1)=v_i\overline{v_j}$.  Our result now follows from Theorem \ref{comex}.
\end{proof}



\begin{corollary}\label{coro}
Let $C$ be a rank $r$ complex correlation matrix.  Then $\frac{1}{r}C+\frac{r-1}{r}I$ is in the convex hull of complex rank one correlation matrices.
\end{corollary}

\begin{proof}  We must have vectors $\{v_k\}_{k=1}^{r}$ such that $C=\sum_{k=1}^r v_kv_k^*$; then all these vectors must satisfy  $\Vert v_k\Vert_{\infty}\le 1$.  Since $\frac{1}{r}C+\frac{r-1}{r}I=\frac{1}{r}\sum_{k=1}^r M(v_k)$, our result now follows from Corollary \ref{corold}.
\end{proof}

\begin{remark}
{\rm 
We note that this can be viewed as the complex version of \cite[Theorem 7]{o2013self} which gave an identical result for real correlation matrices. It was observed in \cite{o2013self} that \cite[Theorem 7]{o2013self} is not optimal at least in small dimensions, and it is likely that the same is true for Corollary \ref{coro}.  We note that any extreme point of the set of $d \times d$ correlation matrices has rank at most  $\lfloor \sqrt{d} \rfloor$ and hence for any $d \times d$ complex correlation matrix $C$, we have that $\frac{1}{\lfloor \sqrt{d} \rfloor}C+\frac{\lfloor \sqrt{d} \rfloor-1}{\lfloor \sqrt{d} \rfloor}I$ is in the convex hull of complex rank one correlation matrices.  This result is an improvement on \cite[Proposition 4.1]{dykema-juschenko}.
}
\end{remark}



\section{Acknowledgements}

We thank John Watrous for communicating Example~\ref{watrouseg} to us.
D.W.K. was partially supported by the NSERC Discovery Grant RGPIN-2018-400160. R.P. was partially supported by the NSERC Discovery Grant RGPIN-2022-04149.
 M.R is supported by the European Research Council (ERC Grant Agreement No. 851716).

\bibliography{KLPR}
\bibliographystyle{plain}

\end{document}